\documentclass[12pt]{article}
\usepackage{color}

\usepackage[numbers,sort&compress]{natbib}
\usepackage{graphicx,latexsym,euscript,makeidx,color,bm}
\usepackage{amsmath,amsfonts,amssymb,amsthm,thmtools,mathtools,mathrsfs,enumerate}
\usepackage[colorlinks,linkcolor=blue,anchorcolor=green,citecolor=blue]{hyperref}

\usepackage{graphicx}
\usepackage{amsmath}
\usepackage{mathrsfs,amsfonts}
\usepackage{amsfonts}
\usepackage{amsmath,amssymb,amsfonts}
\usepackage{amssymb}
\usepackage{color}
\usepackage{amsmath}  
\usepackage{amsfonts} 
\usepackage{graphicx} 
\usepackage{braket}
\usepackage{enumerate}
\usepackage{indentfirst}
\usepackage{longtable}
\usepackage{setspace}

\usepackage{geometry}
 \def\sB{\mathscr{B}}

\def\dbE{\mathbb{E}}     
\def\dbF{\mathbb{F}}   \def\cF{{\cal F}}

\def\dbQ{\mathbb{Q}}     
\def\dbR{\mathbb{R}}

   \def\cU{{\cal U}}

\def\Om{\Omega}

\def\ms{\medskip}

        \def\q{\quad}                      
                        
    \def\hb{\hbox}                     
                   
         \def\rf{\eqref}                    
  \def\deq{\triangleq}               
            \def\({\Big (}
\def\les{\leqslant}                  \def\){\Big )}
\def\leq{\leqslant}       \def\geq{\geqslant}
          \def\[{\Big[}
           \def\]{\Big]}
                   \def\cd{\cdot}

        \def\ts{\times}                      

               \def\si{\sigma}
\def\e{\epsilon}

\theoremstyle{plain}

\renewcommand\theequation{\arabic{equation}}

\makeatletter
\renewcommand{\theequation}{%
	\thesection.\arabic{equation}}
\@addtoreset{equation}{section}
\makeatother

\def\ms{\medskip}

        \def\q{\quad}                      
                        
    \def\hb{\hbox}                     
                   
         \def\rf{\eqref}                    
  \def\deq{\triangleq}               
            \def\({\Big (}
\def\les{\leqslant}                  \def\){\Big )}
\def\leq{\leqslant}       \def\geq{\geqslant}
          \def\[{\Big[}
           \def\]{\Big]}
                   \def\cd{\cdot}

        \def\ts{\times}                      

               \def\si{\sigma}
\def\e{\epsilon}

\def\cU{{\cal U}}

\def\de{{\delta}}
\def\ep{{\epsilon}}\def\ga{{\gamma}}
\def\si{{\sigma}}
\def\ze{{\zeta}}

\def\Om{{\Omega}}

\def\<{\left<}\def\>{\right>}\def\({\left(}\def\){\right)}

\def\Om{{\Omega}}
\def\<{\left<}\def\>{\right>}\def\({\left(}\def\){\right)}

\def\de{{\delta}}
\def\ep{{\epsilon}}\def\ga{{\gamma}}
\def\si{{\sigma}}
\def\ze{{\zeta}}

\def\Om{{\Omega}}

\def\<{\left<}\def\>{\right>}\def\({\left(}\def\){\right)}

  \def\deq{\triangleq}               
            \def\({\Big (}
\def\les{\leqslant}                  \def\){\Big )}
\def\leq{\leqslant}       \def\geq{\geqslant}
\def\le{\leqslant}    
          \def\[{\Big[}
           \def\]{\Big]}
                   \def\cd{\cdot}

        \def\ts{\times}

\newtheorem{theorem}{Theorem}[section]
\newtheorem{definition}[theorem]{Definition}
\newtheorem{proposition}[theorem]{Proposition}

\newtheorem{lemma}[theorem]{Lemma}
\newtheorem{remark}[theorem]{Remark}

\makeatletter
\renewcommand{\theequation}{%
	\thesection.\arabic{equation}}
\@addtoreset{equation}{section}
\makeatother

\allowdisplaybreaks[4]

\begin{document}

	\title{\bf Near optimal controls for partially observed stochastic linear quadratic problems\thanks{This work was supported by
			the National Key R\&D Program of China (2022YFA1006102), 
			the National Natural Science Foundation of China (12271242, 12322118, 12101291 and 12471418), Guangdong Basic and Applied Basic Research Foundation (2025B151502009), and Shenzhen Fundamental Research General Program (JCYJ20230807093309021).}}
	
	\author{
		Jingrui Sun\thanks{Department of Mathematics and Shenzhen International Center for Mathematics,
			Southern University of Science and Technology, Shenzhen 518055, China (sunjr@sustech.edu.cn).},
		 Jiaqiang Wen\thanks{Department of Mathematics and Shenzhen International Center for Mathematics, Southern
			University of Science and Technology, Shenzhen, Guangdong, 518055, China (wenjq@sustech.edu.cn).} ,
		 Jie Xiong\thanks{Department of Mathematics and Shenzhen International Center for Mathematics,
			Southern University of Science and Technology, Shenzhen 518055, China
			(xiongj@sustech.edu.cn).},
		 Wen Xu\thanks{Department of Mathematics,
			Southern University of Science and Technology, Shenzhen 518055, China
			(12231279@mail.sustech.edu.cn).
	}}
	
	\date{}

	\maketitle
	
	\noindent\bf Abstract. \rm
  In this article, we consider a stochastic linear quadratic  control  problem  with partial observation. A near optimal control in the weak formulation
  is characterized. 
		The main features of this paper are the presence of the control  in the diffusion term of the state equation;   
		the circular dependence between the control process and the filtration generated by the  observation;   and the  observation  process contains an unbounded drift term. 
	We address these difficulties  by  first restricting  the control to a smaller domain,  which enables us to apply the 
	 Girsanov theorem using a conditional argument and thereby break the circular dependence. Subsequently, we study the restricted problem using a non-standard 
	 variation method.  
The  desired near optimal control is then obtained by taking  the limit of an approximating sequence.



\section{Introduction }
\renewcommand{\theequation}{\thesection.\arabic{equation}}

Linear-quadratic (LQ) optimal control problem plays a fundamental role in control theory, which appeared with the birth of stochastic analysis and developed rapidly in recent decades due to its wide range of applications. The study of stochastic linear-quadratic (SLQ) optimal control problems can be traced back to the work of Kushner \cite{Kushner-62} and Wonham \cite{WH}.
In the classical framework, the SLQ optimal control problem can be solved elegantly via the Riccati equation under some mild conditions on the weighting coefficients (see Yong and Zhou \cite[Chapter 6]{Yong2012x}). In other words, 
the optimal control could be constructed if the controller has access to the exact value of the state. 
However, in many practical situations, it often happens that some components of the state may be inaccessible for observation, and there could be noise existing in the observation systems. In this case, only partial information is available to the controller, and the control has to be selected according to the information provided by the observation systems.
This class of  problems is summarized as SLQ optimal control problem with partial observation.

The partial observation problem in stochastic control involves how to make optimal decisions in a stochastic environment with incomplete information. To our best knowledge,  the earliest research on optimal control with partial observation dates back to Florentin \cite{F1962}. Since then,   researchers found that  it  can be  applied to many practical problems such as  the mean-variance portfolio selection problems (see Xiong $et\ al$ \cite{XiongJieyu2021}, and Xiong and Zhou \cite{Xiongyu2007}).  
It should be pointed out that the key idea in dealing with the partial observation problem is to separate it into a control and a  filtering problem. On the one hand, for the latter, the filtering theory will be very useful.
Note that the objective of the filtering theory is to obtain the best estimate of the state process based on noisy observations related to this state.  In general, this estimation is equivalent to finding a suitable conditional expectation.  For more research on filtering theory, we refer the reader to Bode and Shannon \cite{Bodenon1950}, Fujisaki {\em et al}  \cite{Fujisakikuni1972},  Zakai \cite{Zakai1969}, and Zhang and Xiong \cite{Zhang2024}.   As a systematic introduction to stochastic filtering, one can refer to the books of Kallianpur \cite{Kallianpur}  and Xiong \cite{xiong}.   
 On the other hand, for the former,  compared to the case of completely observable, the essential difficulty of the partial observation problem is that the filtration $\cF^Y$ is not fixed
(depending on the control $u$) and the linear structure of the admissible control set is thereby corrupted. For this reason, it is hard to derive the optimality system by the variational method.

In general, there are three main approaches to address these  difficulties.  The first one is summarized as the {\it separation principle}. By applying the separation principle one can  decouple the problems of optimal control and state estimation. 
Note that this method needs certain extra constraints on the admissible control $u(\cd)$, and the diffusion term of the state process  does not include the control process $u(\cd)$;  for more details, see Wonham \cite{WH},  Davis \cite{DA}, and   Snyder {\em et al.} \cite{Snyder}.
Along this line, much research on LQ optimal control with partially observable systems has been studied  in recent years, among which we would like to mention
 the works of Huang $et\ al$  \cite{Huang-Wang-Xiong2009},  Huang $et\ al$ \cite{Huang-Wang-Zhang2020}, and Wang $et\ al$
  \cite{Wang-Wang-Yan2021} on backward stochastic control systems, 
the works of  Shi and Wu  \cite{Shi-Wu2010} and Wang $et\ al$ \cite{Wang-Wu-Xiong2015} on optimal control problems of forward-backward stochastic differential
equations (FBSDEs), 
and the works  of Wang and Yu \cite{Wang-Yu2012}, Wu and Zhang \cite{Wu-Zhuang2018} on differential games.
The second method is summarized as the {\it Girsanov transformation}.  By applying  Girsanov's transformation, one could convert the partial observation problem into a partial information one. 
Note that this approach requires higher integrability of the observation function, such as the boundedness of this function, which cannot be fulfilled in the LQ case in general.
For more details, see  Bensoussan and Schuppen \cite{Ben}, Wang $et\ al$  \cite{Wang},
Li and Tang \cite{Li},  and the references therein. 
The third method is called the {\it orthogonal decomposition}, recently  introduced
by Sun  and Xiong  \cite{SX}, where  they  solve an SLQ optimal control problem with partial observation   but with the control  $u(\cd)$ vanished in the diffusion term of the state process. To summarize, the relevant literature mentioned above has at least one of the following limitations:  (i) Additional restrictions on the control $u(\cdot)$  are required; (ii) Higher integrability of the state process  $X(\cdot)$  is necessary; (iii)  The diffusion term of the state process  does not depend on the control $ u(\cdot) $.

In this paper, we consider a stochastic LQ control problem with partial observation, where the control $u(\cdot)$  appears in the diffusion term of the state process. The aim is to find a
near optimal control in the weak formulation.  The main contribution of this paper is to push the research boundary
 of the partial observation stochastic  LQ  control problem to the case that the previous three methods are not directly 
 available. Firstly, it does not satisfy the separation principle. Secondly, since the drift term of the observation model is not bounded, 
 the Novikov 
condition is not satisfied. Therefore, the Girsanov theorem is not directly applicable. Thirdly, the diffusion coefficient of the state equation is not constant so the orthogonal decomposition method is not available.  We study the  problem  by restricting the control domain   to a subset and verifying the condition for the  Girsanov theorem using a conditional argument. Then, we prove that the {\color{red}Radon-Nikodym} derivative has an
 order $1+k$ moment for a small constant $k>0$. After that, because of the lacking of the second moment, a non-standard convex variation method is applied to obtain the  stochastic maximum principle (SMP) for  this restricted problem.   Finally, a near optimal control  is  obtained  from an  approximate sequence.

The remainder of this paper is organized as follows.  Section \ref{se2} formulates the control problem
under partial observation in the weak formulation. Section \ref{se3} is  dedicated to breaking the circular dependence under the control restriction 
by a conditional argument to verify the condition of the Girsanov theorem.  In Section \ref{se4}, we establish higher-order moment estimates  and 
use a non-standard variation method to  derive  the optimal control for the restricted problem.  Furthermore,  we obtain a near optimal control for the original problem which is the main result of this article.  In Section \ref{se5}, we conclude with some remarks.

\section{Problem formulation}\label{se2}

 Let $\mathbb{S}^n$ be the collection of the symmetric $n\times n$ matrices,  $\mathbb{I}_d$ stand for the identity matrix of size $d\times d$, and $\mathfrak{1}_{\mathbb{A}}(x)
 $ denote the indicator function of the set $\mathbb{A}$. For $M$ and $N$ in $\mathbb{S}^n$, we write $M \geq N $ (resp., $M > N$)
if $M - N$ is positive semi-definite (resp., positive definite).   $ \sB(C([0,T],\mathbb{R}^d))$  denotes the Borel $\si$-algebra generated by the $\mathbb{R}^d$ valued continuous functions. Let  $W = \{W(t) ; t \geq 0\}$  and $\widetilde{W} = \{(\widetilde{W}_1(t), ..., \widetilde{W}_d(t))^\top; t \geq 0\} $ be  two standard  independent Brownian motions with values in $\dbR$ and $\dbR^d$, respectively.   The superscript $\top$ denotes the transpose of a vector or a matrix, all vectors involved in this paper are column vectors.  The notation $\mathbb{E}^{\mathbb{P}}$ represents the expectation under the probability $\mathbb{P}$,   and $\mathbb{E}_{W}$   represents the conditional expectation given the Brownian motion $W$. 

We consider the following  controlled linear stochastic differential equation (SDE)  over a finite time  horizon $s \in [0, T]$:
\begin{equation}\label{stateequ}
	\left\{\begin{aligned}
		d X(s)=&\ [  \left.A(s) X(s)+B(s) u(s)\right] d s 
		+C(s) u(s)  d W(s),  \\
		X(0)= &\ x,
	\end{aligned}\right.
\end{equation}
where the initial state $x \in\mathbb{R}^n$ is fixed.
For the state equation \rf{stateequ}, we consider   the  following linear SDE describing the observation process:
\begin{equation}\label{obss}
	\left\{\begin{aligned}
		dY(s)=&\ H(s)X(s)ds+d \widetilde{W}(s), \quad s \in[0, T], \\
		Y(0)= & \ 0.
	\end{aligned}\right.
\end{equation}

Throughout this article, we impose the following assumptions for the coefficients of equations \rf{stateequ}-\rf{obss}.

{\em \begin{itemize}
	\item [\textbf{(A1)}] The following deterministic functions 
	\begin{align*}
		A:[0,T]\to\dbR^{n\times n}, \q~ {B:[0,T]\to\dbR^{n\times d}, }  \q~
		C:[0,T]\to\dbR^{n\times d}, \q~ H:[0,T]\to\dbR^{d\times n} 
	\end{align*}
	\end{itemize}
	are  Lebesgue measurable and bounded on $[0, T]$.}

 Note that the observation process $Y(\cd)$ is a square-integrable  $\dbF$-semimartingale. 
Meanwhile, let    $\dbF^Y=\{\cF^Y_s;0\les s\les T\}$ be the usual augmentation of the filtration generated by $Y(\cdot)$. Clearly, the filtration $\dbF^Y$   depends on the choice of the control $u(\cdot)$.  Additionally, we say that a control $u(\cdot)$ is admissible if it is $\dbF^Y$-adapted and square-integrable. In other words, we arrive at the following admissible set:
$$\cU_{ad} = \Big\{ u(\cd): \Omega\ts [0, T]  \to \mathbb{R}^d  \bigm|
u(s) \in \cF^Y_s \ \ \hb{for all } s\in[0,T] \ \text{and} \  \mathbb{E}\int^T_0 |u(s)|^2ds<\infty\Big\}.$$ 
Define the cost functional as follows
\begin{equation}\label{cost}
	J\big(u(\cdot)\big)=\mathbb{E}\bigg\{\int^T_0\big[ \langle G(t)X(t), X(t) \rangle+\langle R(t)u(t), u(t) \rangle \big]dt+\langle G_1X(T), X(T) \rangle  \bigg\},
\end{equation}
where $\langle \cdot , \cdot \rangle$ is the Frobenius inner product.  Throughout this article, we make the following assumptions regarding the coefficients of the cost functional (\ref{cost}). 

\noindent{\em {\textbf {(A2)}}  The following deterministic functions 
	$$
G(\cdot):[0, T] \to  \mathbb{S}^n, \q~  \; R(\cdot):[0, T] \to  \mathbb{S}^d
	$$
	are  Lebesgue measurable and bounded on $[0, T]$ with
	$$
	 G(\cdot) \geq 0,\q~ \; R(\cdot) \geq \delta  \mathbb{I}_d, \q~ \; a.e. \  \  t \in [0, T],
	$$
where $\de>0$ is a constant. Moreover,	$ G_1 \in \mathbb{S}^n $ and $G_1 \geq 0 $.}

The  aim of this article is to find controls  $\tilde u(\cdot) \in \cU_{ad} $  under  the weak formulation  to minimize the cost functional (\ref{cost}) in the  sense of near optimal.   The control problem under weak formulation is stated in the book of   Yong and Zhou  \cite[Chapter 2]{Yong2012x}.  Next,  we state the  definition of weak formulation for stochastic LQ  control  problem  with partial observation as follows.
\begin{definition}\label{de11}
	A 7-tuple $(\Omega, \mathcal{F},  \mathbb{P}, \mathbb{F}, W(\cdot), \widetilde W(\cdot), u(\cdot))$ is called a weak admissible control, and $(X(\cdot), Y(\cdot),  u(\cdot))$ a weak admissible triple, if 
	
	(a) $(\Omega, \mathcal{F},  \mathbb{P}, \mathbb{F})$ is a filtered probability space satisfying the usual
	conditions,  $W = \{W(t) ; t \geq 0\}  \in \mathbb{R}$ and $\widetilde{W} = \{\widetilde{W}(t)=(\widetilde{W}_1(t), ..., \widetilde{W}_d(t))^\top; t \geq 0\}  \in \mathbb{R}^d$  are  standard independent  Brownian motions defined on it,  and  $\mathbb{F} = \{\mathcal{F}_t\}_{t\geq 0}$ is the natural filtration
	generated by $W(\cdot)$ and $\widetilde{W} (\cdot)$,   augmented by all the $ \mathbb{P}$-null sets. 
	
	(b) $u(\cdot) \in \mathcal{U}_{ad}$, and  $X(\cdot)$ (resp. $Y(\cdot)$ )  is the unique weak solution of Eq. (\ref{stateequ}) (resp.  Eq. (\ref{obss}) ) on $(\Omega, \mathcal{F},  \mathbb{P}, \mathbb{F})$  under the control $u(\cdot)$.  
	
	(c)  The cost functional $J(u(\cdot))$ is defined on the  filtered probability space $(\Omega, \mathcal{F},  \mathbb{P}, \mathbb{F})$ 
	associated with the 7-tuple.  
\end{definition}

The set of all {\color{red}weak admissible} controls is denoted by $\mathcal{U}_{a d}^w$. Sometimes, we might write $u(\cdot) \in \mathcal{U}_{a d}^w$ instead of $(\Omega, \mathcal{F},  \mathbb{P}, \mathbb{F}, W(\cdot), \widetilde W(\cdot), u(\cdot))\in \mathcal{U}_{a d}^w$  when it is clear from the context that the weak formulation is under consideration.    The LQ  near optimal control  problem with partial observation under the weak formulation   can be stated as follows.

\ms

\textbf{Problem (SLQ)}.  For any $\epsilon \in (0, 1)$,   find a 
control $\tilde u (\cdot) \in \mathcal{U}^w_{a d}$  such that  the following holds 
\begin{equation}\label{J_tildeu}
	J(\tilde u(\cdot)) \leq  \inf_{u(\cd) \in \mathcal{U}^w_{a d}}J(u(\cdot))+ \epsilon, 
\end{equation}
 If such $\tilde u(\cdot)$ exists, then it is called a near optimal  control, and the corresponding
$X(\cdot)$ to  (\ref{stateequ})  is called the    near optimal trajectory.

{\color{red}\begin{remark}
Optimal control or game problems under partial observation have been studied by many authors, for example, \cite{CNW} and \cite{LNW}. However,
their setups are different from ours. In \cite{LNW}, they assume that the Brownian motion is observable (namely, $H=0$ in our notation). In this case, there is no circular dependence which is the main technical hurdle overcome in the present article.

In \cite{CNW}, they used a backward separation principle to break the circular dependence. If we adapt that principle to the current model, we would have required the control to be $\cF^{Y^0}$-adapted, where $Y^0(\cdot)$ is given by $Y^0(t)=\int^t_0 H(s)X^0(s)ds+\widetilde{W}(t)$, and
$X^0(t)=x+\int^t_0 A(s) X^0(s)d s.$ In this case, $\cF^{Y^0}=\cF^Y$ and the circular dependence is broken. However, we do not make this extra requirement.
\end{remark}}

\section{Breaking the circular dependence}\label{se3}
\renewcommand{\theequation}{\thesection.\arabic{equation}}
In this section,  we  fix a complete filtered probability space $(\Omega, \mathcal{F},  \mathbb{P}, \mathbb{F})$ on which  two   standard independent Brownian motions,  $W (\cdot) $ and $\widetilde{W} (\cdot)$ are defined.  Here,   $\mathbb{F} = \{\mathcal{F}_t\}_{t\geq 0}$ is the natural filtration
generated by $W(\cdot)$ and $\widetilde{W} (\cdot)$,   augmented by all the $ \mathbb{P}$-null sets. 
Then,   we  break the circular dependence  by applying  the  Girsanov theorem  under the control
restriction using a conditional argument. For this
purpose, we need to do some preparations.

First, we restrict the control domain to a subset  of $\mathcal{U}_{a d}$ that
the integration  of the square of the   control process with respect to the time variable  is  bounded
by a constant.  We denote such set of controls by  $\cU^b_{ad}$  , i.e.,
  \begin{equation}\label{ub_ad}
\cU_{ad}^b = \Big\{ u(\cd): \Omega\ts [0, T]  \to \mathbb{R}^d  \bigm|
u(s) \in \cF^Y_s, \ \ \forall\ s\in[0,T] \ \text{and} \  \int^T_0 |u(s)|^2ds \leq K^2_c\Big\},
  \end{equation}
where $K_c$ is a fixed positive constant. 

\ms

For controls from $\cU_{ad}^b$, we first prove that the state process $X(\cdot)$ is bounded by a random variable that  depends only on
the Brownian motion $W(\cdot)$.  
Then, we verify the conditional Novikov condition with $W$ given. 
This then validates the use of Girsanov's theorem. To do this,  we set   $\ze(\cd)$ to 
be the solution of the following ordinary differential equation (ODE):
\begin{equation}\label{ODE}
	\left\{\begin{aligned}
	d\ze(t)=&\  -\ze(t)A(t) dt, \q~ t\in[0,T],\\
		\ze(0)=&\  \mathbb{I}_n.
	\end{aligned}\right.
\end{equation}
Then, by applying It\^o's formula, 	one can show that
$$d\big(\ze(t) X(t) \big) =\ze(t)B(t)u(t) dt + \ze(t) C(t) u(t)  dW(t), \q~  t\in[0,T],$$
which implies that
\begin{equation}\label{solx}
	X(t)= \ze(t)^{-1}x+\ze(t)^{-1}  \int^t_0\ze(s) B(s) u(s) ds+\ze(t)^{-1} \int^t_0  \ze(s)C(s) u(s)dW(s).
\end{equation} 
Next, for the sake of simplicity in presentation, we define
$$M(t)\deq \int^t_0\ze(s)C(s)u(s)dW(s), \q~  t\in[0,T].$$
Additionally, for the entire given path of $\{\widetilde{W}(t); \  0\les t\les T\}$,  we
denote $\mathbb{P}_{\widetilde{W}}$  as its corresponding conditional probability. 

%
\begin{lemma}\label{lemma22} \sl 
Given the whole  path of $\{\widetilde{W}(t); \  0\les t\les T\}$,  the process
 $ M(\cd) $ is a $\mathbb{P}_{\widetilde{W}}$-martingale with quadratic variation process
	$$\< M\>(t)=\int^t_0 | \ze(s) C(s)u(s)|^2ds \equiv \si(t), \q~ t\in[0,T].$$
	\end{lemma}
	\begin{proof}
	 For any $\mathbb{A} \in \cF_s$ and  $\mathbb{B} \in \sB(C([0,T],\mathbb{R}^d))$,  we have
	\begin{align*}
&\ \mathbb{E} ^{\mathbb{P}}\bigg\{
\mathbb{E}^{\mathbb{P}_{\widetilde{W}}} \Big[\int^t_s \ze(r)C(r)u(r)dW(r)\mathfrak{1}_{\mathbb{A}} \Big]\mathfrak{1}_{\widetilde W \in \mathbb{B}}\bigg\} \\
	&=\ \mathbb{E}^{\mathbb{P}} \bigg\{\int^t_s \ze(r)C(r)u(r)dW(r)\mathfrak{1}_{\mathbb{A}} \mathfrak{1}_{\widetilde W \in \mathbb{B} }\bigg\} \\
	&=\  \mathbb{E}^{\mathbb{P}} \bigg\{ \int^t_s \ze(r)C(r)u(r) \mathfrak{1}_{\mathbb{A}} \mathfrak{1}_{\widetilde W \in \mathbb{B}} dW(r)\bigg\} =\ 0,
	\end{align*}
	which implies that  the process $M(\cd)$ is a $\mathbb{P}_{\widetilde{W}}$-martingale.  Moreover, it is trivial to verify the identity  of the quadratic variation.
	\end{proof}	
	Based on the above result,  we can now prove that the norm of the 
	state process $X(\cd)$ has an upper bound, which   plays an important role in  the following   \autoref{lemma2.1}. 
\begin{lemma}\label{lemma} \sl 
There exists a random variable $K (W)$, which  is independent of $\widetilde{W}$ such that for any $u (\cdot) \in \cU^b_{ad}$, 
$$\sup_{0\leq t\leq T}|X(t)| \leq K(W).$$
\end{lemma}

\begin{proof}
It is clear that there exists a constant  $K$   such that the first two terms of (\ref{solx})  are bounded by it.    We claim that the third term of (\ref{solx}) can  be bounded by a random variable  $\tilde{K}(W)$,    which is independent of the Brownian motion $\widetilde{W}$.   Since $u(\cdot) \in \cU^b_{ad}$, 
$M(\cdot)$ is not independent of $\widetilde{W}$. By Lemma \ref{lemma22},  $M(\cdot)$ is a $\mathbb{P}_{\widetilde{W}}$-martingale.  

Let $\widetilde{M}(t)=M(\si^{-1}(t))$. Then, $\widetilde{M}(\cdot)$ is a $\mathbb{P}_{\widetilde{W}}$-martingale with quadratic variation process $\langle\widetilde{M}\rangle(t)=t$. Hence, $\widetilde{M}(\cdot)$ is a $\mathbb{P}_{\widetilde{W}}$-Brownian motion. Namely, for any path of $\widetilde{W}(\cdot)$, the distribution of  $\widetilde{M}(\cdot)$  under  $\mathbb{P}_{\widetilde{W}}$ is not changed,   which implies that $\widetilde{M}(\cdot)$ is independent of ${\widetilde{W}}(\cdot)$.  Note that
$M(t)=\widetilde{M}(\si(t))$, and
\begin{equation*}\si(t)\le c K_c^2\|\ze\|^2\equiv\si^0,\end{equation*}
where $\|\ze\|=\sup_{0 \leq t\le T}|\ze(t)|$, $c > 0$ is a constant,  so  $\sigma^0$ is a  positive constant.  Further, we have
\begin{equation*}|M(t)|\le\sup_{0\leq s\le \si^0}|\widetilde{M}(s)|.\end{equation*}
Note that   $\sup_{0\leq s\le \si^0}|\widetilde{M}(s)|$ is a random variable depending on $W$ and $\widetilde{W}$. 
Let 
 \begin{equation*}
 \tilde{K}(W, \widetilde W) \deq \sup_{0\leq s\le \si^0}|\widetilde{M}(s)|,
\end{equation*}
and
$$
g (\widetilde W) \deq \mathbb{E}_{ \widetilde  W} \(\tilde{K}(W, \widetilde W)h(W)\), 
$$
where $h$ is a bounded functional.
Then for any bounded functional $f(\cdot)$,  we  have 
\begin{equation*}
	\begin{aligned}
		\mathbb{E}\big[g (\widetilde W) f(\widetilde{W})\big] =&    \mathbb{E} \big[ \mathbb{E}_{\widetilde{W}}\( \tilde{K}(W, \widetilde W)h({W})\)  f(\widetilde{W})    \big] \\
				=&  \mathbb{E}\big[\tilde{K}(W, \widetilde W)h(W) f(\widetilde{W})\big]\\
		=&   \mathbb{E}\big[\tilde{K}(W, \widetilde W)h(W) ]  \mathbb{E}\big[ f(\widetilde{W})\big]\\
		=& \mathbb{E}\big[g (\widetilde W)\big]  \mathbb{E}\big[f (\widetilde W)\big].
	\end{aligned}
\end{equation*}
Namely,   $g (\widetilde W)$    is independent of $\widetilde W$, which further implies that $g (\widetilde W)$   is a constant.  Therefore, $\tilde{K}(W, \widetilde W)$ does not depend on $\widetilde W$. Combining the above, we obtain the desired conclusion.
\end{proof}

Based on the  above results, we  define
\begin{equation*}
	\begin{aligned}
		N(t)=& \exp\bigg\{-\int^t_0 \( H(s)X(s)\)^\top d\widetilde{W}(s) -\frac12\int^t_0 |H(s)X(s)|^2 ds\bigg\},\q~ t\in[0,T], 
	\end{aligned}
\end{equation*}
which,   using It\^o's formula, can be expressed as follows:
$$N(t) = 1- \int_0^t N(s)\(H(s)X(s)\)^{\top}d\widetilde{W}(s),\q~ t\in[0,T].$$
Therefore, $N(\cd)$ is a one-dimensional continuous  local martingale.

\begin{proposition}\label{lemma2.1} \sl 
	$N(\cdot)$ is a martingale under the probability $\mathbb{P}$. 
\end{proposition}
\begin{proof}
We only  need to prove that $\mathbb{E}N(t)=1$. 
From  \autoref{lemma},   we  have 
$$\mathbb{E}_{W}\bigg\{ \exp\(\frac12\int^T_0 |X(s)|^2ds\)\bigg\}
\le\mathbb{E}_{W}\bigg\{ \exp\(\frac12 K^2 (W) T\)\bigg\}
= \exp\(\frac12 K^2 (W)T\) <\infty, \quad a.s..$$
Applying  the Girsanov theorem with the conditional Novikov condition satisfied, we get $\mathbb{E}_{W} N(t)=1$. Therefore,
$$\mathbb{E} N(t)=\mathbb{E} \Big\{\mathbb{E}_{W} \{N(t)\}\Big\}=1,\q~ t\in[0,T].$$
This completes the proof.
\end{proof}

Now, based on the preparations  above,  we can  break the circular dependence between the filtration $\dbF^Y$ and the control $u(\cdot) \in  \cU^b_{ad}$. 
As noted from  \autoref{lemma2.1}, for  any control $u(\cdot) \in  \cU^b_{ad}$,  we can apply Girsanov's theorem to transform the observation process $Y(\cdot)$  into a standard Brownian motion under a new probability measure $\mathbb{Q} $, which is defined by
$$\frac{d\mathbb{Q}}{d\mathbb{P}}\bigg|_{\cF_t}=N(t),\q~ t\in[0,T].$$
Then, the pair  $\big(Y(\cdot)^\top,W(\cdot)\big)^\top$ is a $(d+1)$-dimensional Brownian motion under the probability measure $\mathbb{Q}$. 

Let  $Z(\cd)=N^{-1}(\cd)$. It is then easy to verify that $Z(\cd)$ satisfies
		\begin{equation*}
			\begin{aligned}
				dZ(t) = &  Z(t) \Big\{(H(t)X(t))^{\top}d\widetilde{W}(t)+| H(t)X(t) |^2dt \Big\},\q~ t\in[0,T].
			\end{aligned}
		\end{equation*}
		Namely, 
		\begin{equation}\label{zzz}
			\begin{aligned}
				dZ(t)&=Z(t) \langle  H(t)X(t), dY(t)\rangle, \q~ t\in[0,T].
			\end{aligned}
		\end{equation}

Now,   the circular dependence is broken. Therefore, our next primary objective is to  establish SMP over the  control set $ \cU^b_{ad}$ under the weak formulation. 
Note that in this special case, 
the state process becomes a $(n+1)$-dimensional process $\big(X(\cdot), Z(\cdot)\big)$. Moreover,   the system is described by the following SDEs: 
\begin{equation}\label{well}
	\left\{
	\begin{aligned}
	dX(t)=&\big(A(t) X(t)+B(t)u(t)\big)dt+C(t)u(t)dW(t), \\
	dZ(t)=&Z(t)\big(H(t)X(t) \big)^{\top}  dY(t), \q~ t\in[0,T], \\
	X(0)= &x, \; Z(0) =1.
	\end{aligned}
	\right.
\end{equation}
The related  cost functional is given by
\begin{equation}\label{eq1010a}
 J(u(\cdot))=\mathbb{E}^{\mathbb{Q}}\bigg\{\int^T_0\big( \langle G(t)X(t), X(t) \rangle+\langle R(t)u(t),u(t)\rangle \big) Z(t)dt
	+ \langle G_1X(T), X(T) \rangle Z(T)\bigg\}.
\end{equation} 
We label the above problem as
Problem (O).

\textbf{Problem (O)}.  Find a  control $ u^* (\cdot)\in\cU^b_{ad}$   such that
$$
 J( u^*(\cdot))= \inf_{u \in \cU^b_{ad} }J(u(\cdot)). 
$$

\section{Stochastic maximum principle }\label{se4}
\renewcommand{\theequation}{\thesection.\arabic{equation}}

{\color{red}In this section, we first prove the existence of an optimal control $u^*(\cdot) \in \mathcal{U}^b_{\text{ad}}$  that minimizes the cost functional
 $(\ref{eq1010a})$ in the weak formulation. Then, we  deduce the SMP using the convex variation technique for 
Problem (O). Finally, we derive the  near optimal controls for the original problem.}

\subsection{\color{red}Existence of an optimal control}

	Recall that $Z(\cdot)$ is defined by (\ref{zzz}). It is clear that $Z(\cdot)$ is  a  $\mathbb{Q}$-martingale.
Although
we have  $\mathbb{E}^{\mathbb{Q}}[Z(\cd)] < \infty$, it is not enough for the variational estimation of the cost functional (\ref{eq1010a}).  Actually, we need its moments of order higher than 1 which is obtained in the following
lemma. 

	\begin{lemma}\label{lem41} \sl 
		There exists a constant $k\in(0,1)$ such that
		$\mathbb{E}^{\mathbb{Q}}[ Z^{1+k} (t)]< \infty$.  
	\end{lemma}
	
	\begin{proof}
		From \rf{zzz}, we have 
	\begin{equation*}
		\begin{aligned}
			Z(t)=&\  \exp\bigg\{\int^t_0  \(H(s)X(s)\)^\top dY(s)  -\frac12\int^t_0 | H(s)X(s)|^2    ds\bigg\},\q~ t\in[0,T]. 
		\end{aligned}
	\end{equation*}
	Therefore, 
	\begin{equation}\label{z^k}
		\begin{aligned}
			Z^{1+k}(t)=&\ \exp\bigg\{(1+k)\int^t_0  \(H(s)X(s)\)^\top  dY(s) -\frac{(1+k)^2}{2}\int^t_0 |H(s)X(s)|^2    ds\\
			&\ \quad \quad\  +\frac{1}{2}(1+k)k\int^t_0 | H(s)X(s)|^2    ds\bigg\},\q~ t\in[0,T]. 
		\end{aligned}
	\end{equation}
Define a new measure $\mathbb{\tilde Q}$ by 
	\begin{equation*}
		\begin{aligned}
			\frac{d\mathbb{\tilde Q}}{d\mathbb{Q}}\bigg|_{\cF_t}= &\  \exp\bigg\{(1+k)\int^t_0  \(H(s)X(s)\)^\top dY(s) -\frac{(1+k)^2}{2}\int^t_0 |H(s)X(s)|^2    ds\bigg\},\q~ t\in[0,T].
		\end{aligned}
	\end{equation*}	
Similar to the proof of Lemma \ref{lemma}, we can show that
$$|X(t)|\le K\Big\{1+\sup_{t\le\si^2}|\widehat{W}(t)|\Big\},$$
where $\si,\ K$ are two constants, and $\widehat{W}(\cdot)$ is a $\mathbb{\tilde Q}$-Brownian motion.

Applying the   Fernique theorem (see Kuo \cite[Theorem 3.3.1 on page 159]{HH1975}), we obtain that for any $t\in[0,T]$, 
	\begin{equation*}\label{fin}
		\begin{aligned}
			\mathbb{E}^{\mathbb{Q}}	[Z^{1+k}(t)] =& \	\mathbb{E}^{\mathbb{\tilde Q} }\bigg\{
			\exp\Big\{\frac{1}{2}(1+k)k\int^t_0 | H(s)X(s)|^2    ds\Big\}\bigg\} \\
			\leq &\ \mathbb{E}^{\mathbb{\tilde Q} }\Big\{\exp\Big(k(1+\sup_{t\le\si^2}|\widehat{W}(t)|^2 )
			\Big)\Big\}\\\
			<&\  \infty,
		\end{aligned}
	\end{equation*}
by taking 	 $k$ small enough. 
\end{proof}

	\ms
	
	Let $ \cU^{b}_{w}$ be the set of weak admissible controls  $(\Omega, \mathcal{F},  \mathbb{F}, \mathbb{P}, W(\cdot), Y(\cdot), u(\cdot))$   such that $u(\cdot) \in \cU^{b}_{ad}$.   
Now we are ready to present a result to establish the existence of the optimal control of Problem (O) in the weak formulation. 
\begin{theorem}\label{thm42}
	Let Assumptions A1 and A2 hold.  Then Problem (O) has a weak optimal control $u^*(\cdot) \in  \cU^b_{w}$. 
\end{theorem}

\begin{proof}
 Let $u_n(\cdot) \in  \cU^b_{ad} $ be a sequence of admissible controls such that $$J(u_n(\cdot))\to \inf_{u\in\cU^b_{ad}}J(u(\cdot)).$$
We can regard $\hat{\Om}\equiv[0,T]\times\Om$ as a probability space with probability measure
$\frac{1}{T}dt d\mathbb{Q}$. Then, $u_n(\cdot)$ can be regarded as a sequence of random variables on $\hat{\Om}$ with
uniformly bounded second moments. We take a subsequence such that $u_{n_k}(\cdot)$ converges in distribution to $u^*(\cdot)$. 
Without loss of generality, we assume that $u_n(\cdot)$ converges in distribution to $u^*(\cdot)$. By Skorohod's representation, 
changing the probability space if necessary,   we may and will assume that  $u_n(\cdot)\to u^*(\cdot)$ a.s..  By the uniform boundedness of the second-integral in $t$-variable of $u_n(\cdot)$, we have 
for $p>0$ $$\mathbb{E}^\mathbb{Q}\(\(\int^T_0|u_n(s)-u^*(s)|^2ds\)^p\)\to 0.$$
Let $(X^*(\cdot), Z^*(\cdot))$ be the solution of system (\ref{well})  corresponding to the control $u^*(\cdot)$, and $(X_n(\cdot), Z_n(\cdot))$  the state process with respect to $u_n(\cdot)$.  Then it is  easy to verify that
$$\mathbb{E}^\mathbb{Q}\sup_{s\in[0,T]}|X_n(s)-X^*(s)|^p\to 0.$$
Applying Lemma \ref{lem41} and choosing $\epsilon \in (0, k)$,  we have
\begin{align*}
	&	\mathbb{E}^{\mathbb{Q}}	\sup_{s\in[0,T]}|Z_n(s)-Z^*(s)|^{1+\ep}\\
	&\le\ K	\mathbb{E}^{\mathbb{Q}}	\sup_{s\in[0,T]}\((Z_n(s)+Z^*(s))^{1+\ep}\left|\int^s_0((HX_n)^\top-(HX^*)^\top)dY\right|^{1+\ep}\)\\
	&\ +K	\mathbb{E}^{\mathbb{Q}}	\sup_{s\in[0,T]}\((Z_n(s)+Z^*(s))^{1+\ep}\int^s_0||HX_n|^2-|HX^*|^2|^{1+\ep}dt\)\\
	&\le\ K\(	\mathbb{E}^{\mathbb{Q}}	\sup_{s\in[0,T]}(Z_n(s)+Z^*(s))^{1+2\ep}\)^{\frac{1+\ep}{1+2\ep}}\\
	&\qquad\times\(	\mathbb{E}^{\mathbb{Q}}	\sup_{s\in[0,T]}\left|\int^s_0((HX_n)^\top-(HX^*)^\top)dY\right|^{\frac{(1+\ep)(1+2\ep)}{\ep}}\)^{\frac{\ep}{1+2\ep}}\\
	&\ +K\(	\mathbb{E}^{\mathbb{Q}}	\sup_{s\in[0,T]}(Z_n(s)+Z^*(s))^{1+2\ep}\)^{\frac{1+\ep}{1+2\ep}}\\
	&\ \qquad\times\(	\mathbb{E}^{\mathbb{Q}}	\sup_{s\in[0,T]}\left|\int^s_0||HX_n|^2-|HX^*|^2|^{1+\ep}ds\right|^{\frac{1+2\ep}{\ep}}\)^{\frac{\ep}{1+2\ep}}\\
	&\le\ K\(	\mathbb{E}^{\mathbb{Q}}	\left|\int^T_0|X_n-X^*|^2ds\right|^{\frac{(1+\ep)(1+2\ep)}{2\ep}}\)^{\frac{\ep}{1+2\ep}}\\
	&\ +K\(	\mathbb{E}^{\mathbb{Q}}	\left|\int^T_0|X_n-X^*|^{2(1+\ep)}ds\right|^{\frac{1+2\ep}{2\ep}}\)^{\frac{\ep}{1+2\ep}}\\
	&\to 0, \;\;  \text{as} \;\; \; n \to \infty.
\end{align*}
These estimates  imply that $(X^*(\cdot), Z^*(\cdot), u^*(\cdot))$ satisfy (\ref{well}) in the weak sense. The optimality of $u^*(\cdot)$ follows from Fatou's lemma. 
\end{proof}

\subsection{\color{red}SMP  for Problem (O)}
	
In this subsection, we present the necessary conditions satisfied by the optimal control	 $u^*(\cdot) \in\cU^b_{ad}$.  To this end,  we
suppose that  $u^*(\cdot)+\e v(\cdot)\in\cU^b_{ad}$ for $\e \in (0, 1)$.

\ms

Let $(X^{\ep }(\cdot),Z^{\ep }(\cdot))$  denote the corresponding solution of SDE \rf{well} with $u(\cd)$ replaced by  $u^*(\cdot)+\ep v(\cdot)$, and let  $(X^*(\cdot), Z^*(\cdot))$ denote the optimal trajectory of SDE \rf{well}  with $u(\cd)$ replaced by the corresponding optimal control  $u^*(\cdot)$. 
The following are the related  variation equations: 
	for $t \in [0, T]$, 
	
	\begin{equation}\label{lem31}
		\left\{
		\begin{aligned}
			dX^1(t)=&\Big\{A(t)X^1(t)+B(t)v(t)\Big\}dt+C(t)v(t)dW(t),\\
			dZ^1(t)=&\Big\{Z^1(t)  (H(t) X^*(t))^{\top} +Z^*(t) (H(t) X^1(t))^{\top}\Big\}dY(t),\\
			X^1(0)= & 0, \; \quad Z^1(0)=0.
		\end{aligned}
		\right.
	\end{equation}
	
	In the following, we  give some basic lemmas  which will be helpful in deriving the variation inequality later.
	
	\begin{lemma}\label{sss42} \sl 
		For any $t\in[0,T]$,
		\begin{equation}\label{3.1}
			X^1(t)=\frac{X^{\ep }(t)-X^*(t)}{\ep},
		\end{equation}
		and
		\begin{equation}\label{32}
			Z^1(t)=\lim_{\ep\to 0}\frac{Z^{\ep }(t)-Z^*(t)}{\ep},
		\end{equation}
		in the sense that
		$$\lim_{\ep\to 0}\mathbb{E}^\mathbb{Q}\left|\frac{Z^{\ep }(t)-Z^*(t)}{\ep}-Z^1(t)\right|^{1+k'}=0,$$
		where $k' \in(0,k)$ and $k$ is given by Lemma \ref{lem41}. 
	As a consequence, we have $\mathbb{E}^{\mathbb{Q}}[|Z^1(t)|^{1+k'}]<\infty$.
	\end{lemma}	
	
	\begin{proof}
		
		Due to the linearity of the state equation with respect to $X$, (\ref{3.1})  is clearly true. 
		Now we prove  (\ref{32}).   Note that 
	\begin{equation*}
		\begin{aligned}
			Z^\epsilon(t)=&\  \exp\bigg\{\int^t_0  \big(H(s)X^\epsilon(s)\big)^\top dY(s)  -\frac12\int^t_0 | H(s)X^\epsilon(s) |^2    ds\bigg\}.
		\end{aligned}
	\end{equation*}
	Therefore,  
	\begin{equation}\label{ww}
		\begin{aligned}
			\lim_{\ep\to 0}\frac{Z^{\ep }(t)-Z^*(t)}{\ep}= &\ Z^*(t) \bigg\{ \int^t_0 (H(s)X^1(s))^{\top} dY(s) - \int^t_0 (H(s)X^*(s))^\top (H(s)X^1(s)) ds  \bigg\}.
		\end{aligned}
	\end{equation}
Denote the right hand side of (\ref{ww}) as  $M^1(t)$.
Applying It\^o's formula, we have 
	\begin{equation*}
		\begin{aligned}
			dM^1(t) =  M^1(t) (H(t)X^*(t))^{\top} dY(t) + Z^*(t)(H(t)X^1(t) )^\top dY(t),
		\end{aligned}
	\end{equation*}
By the uniqueness of the solution to the second equation of (\ref{lem31}), we have $M^1(\cd) = Z^1(\cd)$. 
\textcolor{red}{
Note that  for $k' \in (0, k )$,   we can choose $m \in (1,  \frac{1+k}{1+k'})$,  such that 
\begin{equation}\label{xuxu46}
\begin{aligned}
&\mathbb{E}^{\mathbb{Q}} \Big| Z^*(t) \int ^t_0(H(s)X^*(s))^\top H(s)X^1(s) ds  \Big| ^{1+k'} \\
\leq &K \mathbb{E}^{\mathbb{Q}} \Big[Z^*(t)^{1+k'}  \sup_{s \in [0, t]} |X^*(s)^\top X^1(s)|^{1+k'} \Big] \\
\leq & K\Big\{ \mathbb{E}^{\mathbb{Q}} \Big[Z^*(t)^{(1+k')m} \Big] \Big\} ^{\frac{1}{m}} \Big\{ \mathbb{E}^{\mathbb{Q}} \Big( \sup_{s \in [0, t]} |X^*(s)|^{2(1+k')n} +\sup_{s \in [0, t]} |X^1(s)|^{2(1+k')n} \Big)  \Big\}^{\frac{1}{n}} \\
< &  \infty,
\end{aligned}
\end{equation}
where $n$ is such that $\frac1n+\frac1m=1$.
 Further,  we have 
 \begin{equation}\label{xuxu47}
 \begin{aligned}
 	&\mathbb{E}^{\mathbb{Q}} \Big| Z^*(t) \int^t_0 (H(s)X^1(s))^{\top} dY(s)    \Big| ^{1+k'} \\
 = & \mathbb{E}^{\mathbb{Q}} \Big[Z^*(t)^{1+k'}  \Big|  \int^t_0 (H(s)X^1(s))^{\top} dY(s) \Big|^ {1+k' }  \Big] \\
 	\leq & K\Big\{ \mathbb{E}^{\mathbb{Q}} \Big[Z^*(t)^{(1+k')m} \Big] \Big\} ^{\frac{1}{m}} \Big\{ \mathbb{E}^{\mathbb{Q}}   \Big|  \int^t_0 |X^1(s)|dY(s) \Big|^{(1+k')n} \Big\}^{\frac{1}{n}} \\
 	\leq & K\Big\{ \mathbb{E}^{\mathbb{Q}} \Big[Z^*(t)^{(1+k')m} \Big] \Big\} ^{\frac{1}{m}} \Big\{ \mathbb{E}^{\mathbb{Q}}   \Big( \int^t_0 |X^1(s)|^2ds \Big)^{\frac{1}{2}(1+k')n} \Big\}^{\frac{1}{n}}  \\
 	< &  \infty.
 \end{aligned}
 \end{equation}
Combining  (\ref{xuxu46}, \ref{xuxu47}),  we conclude that $\mathbb{E}^{\mathbb{Q}}[|Z^1(t)|^{1+k'}]<\infty$.   } 
\end{proof}

\textcolor{red}{
By symmetry and Lemma \ref{lem41}, we have the following result.
	\begin{lemma}\label{nlem41} \sl 
	There exists a constant $k\in(0,1)$ such that
	$\mathbb{E}[ N^{1+k} (t)]< \infty$.  
\end{lemma}}

Applying Lemmas  \ref{lem41} and  \ref{sss42},  the optimality of $u^*(\cdot)$ implies
\begin{equation}\label{Ju*}
	\begin{aligned}
0\le& \lim_{\ep\to 0}\frac{J(u^*+\ep v)-J(u^*)}{\ep}\\
=&\mathbb{E}^{\mathbb{Q}}\int^T_0\Big\{(2\langle GX^1,X^* \rangle+2\langle Rv, u^*\rangle )Z^*+(\langle GX^*,X ^*\rangle+\langle Ru^*, u^*\rangle)Z^1\Big\}dt\\
&+\mathbb{E}^{\mathbb{Q}}\Big\{2\langle G_1X^*(T), X^1(T)\rangle  Z^*(T)+\langle G_1X^*(T), X^*(T) \rangle Z^1(T)\Big\}. 
\end{aligned}
\end{equation}
In the above integration, we dropped the time parameters for the simplicity of the presentation.
We will continue to do so in the rest of the paper when it is suitable.

Define the triples $\big(P(\cdot), Q_1(\cdot), Q_2(\cdot)\big)$  and $\big(p(\cdot), q_1(\cdot), q_2(\cdot)\big)$ as  the solutions of the following equations
	\begin{equation}\label{eq0208d}
	\left\{
	\begin{aligned}
		dP(t)=&-\(2Z^*G X^*+A^{\top}P+ Z^* (q_2  H)^\top\)dt +Q_1(t)dW(t)+Q_2(t)dY(t),\; t\in[0,T], \\
		P(T)=&\ 2Z^*(T)G_1X^*(T)
	\end{aligned}
	\right.
\end{equation}
and 
\begin{equation}\label{eq0208e}
	\left\{
	\begin{aligned}
		dp(t)=&-\(\langle GX^*, X ^*\rangle +\langle Ru^*, u^*\rangle+q_2HX^*\)dt+q_1(t)dW(t)+q_2(t)dY(t),\;  t\in[0,T],\\
	 p(T)=&\  \langle G_1X^*(T), X^*(T) \rangle,
	\end{aligned}
	\right.
\end{equation}
respectively.  

{\color{red} Note that it is easy to prove that $\cF^{W,Y}_t = \cF_t$. Then, we present the definition of the solution to BSDEs (\ref{eq0208d}) and (\ref{eq0208e}). 

\begin{definition}\label{def0208a}
We say that $(P,Q_1,Q_2)$ is a solution to BSDE (\ref{eq0208d}) if it is $\cF_t$-adapted and satisfies, for all $t\in[0,T]$, the following equation $\dbQ$-a.s.:
\begin{equation}\label{eq0211c}
	\begin{aligned}
		P(t)=2Z^*(T)G_1X^*(T)+\int_t^T\bigl(2Z^*G X^*+A^{\top}P+ Z^*(q_2H)^\top\bigr)\,ds \\
		-\int_t^T Q_1(s)\,dW(s)-\int_t^T Q_2(s)\,dY(s),
	\end{aligned}
\end{equation}
A similar definition applies to $(p,q_1,q_2)$.
\end{definition}

\begin{lemma}\label{aaas}
	Under Assumptions (A1) and (A2),  BSDEs 
(\ref{eq0208d}) and (\ref{eq0208e})  admit  a unique solution. Moreover,  $(P,Q_1,Q_2)$ is $\ga_1$th-integrable for $\ga_1\in (1,1+k)$  and $(p,q_1, q_2)$ is  $\ga_2$th-integrable
for any $\ga_2>1$  in the following sense:
\begin{equation}\label{eq0211d}\dbE^{\dbQ}\(\sup_{t\in[0,T]}|P(t)|^{\ga_1}+\(\int^T_0\(|Q_1(t)|^2+|Q_2(t)|^2\)dt\)^{\ga_1/2}\)<\infty\end{equation}
and
\begin{equation}\label{eq0211b}
\dbE^{\mathbb{Q}}\(\sup_{t\in[0,T]}|p(t)|^{\ga_2}+\(\int^T_0\(|q_1(t)|^2+|q_2(t)|^2\)dt\)^{\ga_2/2}\)<\infty,\end{equation}
where $k$ is given by Lemma \ref{lem41}.
\end{lemma}
\begin{proof}
Under the original measure  $\mathbb{P}$, we have
	\begin{equation}\label{thmbsde}
	\left\{
	\begin{aligned}
		d p(t)=&-\Big(X^{*\top} GX^* +u^{*\top} R  u^*\Big)dt + q_1(t)dW(t)+q_2 (t)d\widetilde W(t) \\
		p(T)=&X^*(T)^{\top} G_1  X^*(T),
	\end{aligned}
	\right.
	\end{equation}
Note that for any $m >  1$, there exists a constant $K$ such that
$$\dbE\(\sup_{s\in[0,T]}|X^*(s)|^m\)\le K\dbE\(1+\(\int^T_0|u^*(s)|^2ds\)^{m/2}\)\le K\(1+K_c^m\).$$
Hence, by the standard BSDE theory (see, e.g., \cite[Chapter 5]{PR}), the BSDE (\ref{thmbsde}) admits a unique solution. Moreover,  the following estimate holds:
 $$
 \begin{aligned}
 	&\mathbb{E}\left\{\sup_{s \in [0, T] }|p(s)|^{m} +  \left(\int^T_0 |q_1(s)|^2ds\right)^{\frac{m}{2}} + \left(\int^T_0 |q_2(s)|^2ds\right )^{\frac{m}{2}}  \right\}\\
 	& \leq K\mathbb{E} \left(  |X^*(T)|^{2m}+ \Big(\int^T_0 (X^{*\top} GX^*  +u^{*\top} R  u^*) dt\Big)^{m} \right) < \infty. 
 \end{aligned}
 $$
Applying Lemma  \ref{nlem41},  we have 
 $$
 \begin{aligned}
 	\mathbb{E}^{\mathbb{Q }}[\sup_{s \in [0, T] }|p(s)|^{\gamma_2}  ] = &  \mathbb{E}[ N(t)\sup_{s \in [0, T] }|p(s)|^{\gamma_2}] \\
 	\leq &K  \Big(\mathbb{E} N(t)^{1+k} \Big)^{\frac{1}{1+k}}  \Big (\mathbb{E} \sup_{s \in [0, T] }|p(s)|^{\gamma_2 (1+ \frac{1}{k})} \Big)^{\frac{k}{1+k}} < \infty, 
 \end{aligned}
 $$
and 
$$
\begin{aligned}
	\mathbb{E}^{\mathbb{Q }}\left(\int^T_0 |q_1(s)|^2ds\right)^{\frac{\gamma_2}{2}}  =&   \mathbb{E}\left\{N(t)\left(  \int^T_0 |q_1(s)|^2ds\right)^{\frac{\gamma_2}{2}} \right\} \\
	  \leq &  K \Big(\mathbb{E} N(t)^{1+k} \Big)^{\frac{1}{1+k}}   \mathbb{E}\left\{\left(  \int^T_0 |q_1(s)|^2ds\right)^{\frac{\gamma_2}{2}(1+ \frac{1}{k}) } \right\} ^ {\frac{k}{1+k}} < \infty. 
\end{aligned}
$$
Similar arguments apply to $q_2$, which completes the proof of \eqref{eq0211b}. 

Having obtained $q_2$ above, we now proceed to solving BSDE (\ref{eq0211c}). To this end, we need to verify the integrability of the drift term.  Note that by Doob's  inequality,  we have $\dbE^{\dbQ} [\sup_{s \in [0, T]}Z^*(s)^{1+k}]< \infty$. Therefore, 
 for $\eta \in (0, k)$, we can choose $m \in (1, \frac{1+k}{1+\eta})$  such that
\begin{eqnarray*}
&&\dbE^{\dbQ}\left|\int^T_0|2Z^*G X^*+ Z^* (q_2 H)^\top|ds\right|^{1+\eta}\\
&\le&K\dbE^{\dbQ} \left[ \sup_{s\in[0, T]} Z^*(s)^{1+\eta} \(\int^T_0 \( |X^*(s)|^2+|q_2(s)|^2\)ds\)^{\frac{1}{2}(1+\eta)}\right]\\
&\le&K\left[ \dbE^{\dbQ} \sup_{s\in[0, T]} Z^*(s)^{m(1+\eta)} \right]^{\frac{1}{m}} \left[ \dbE^{\dbQ}  \(\int^T_0 \( |X^*(s)|^2+|q_2(s)|^2\)ds\)^{\frac{n}{2}(1+\eta)}\right]^{\frac{1}{n}} \\
&< & \infty.
\end{eqnarray*}
Moreover, 
$$
\dbE^{\dbQ}  |2Z^*(T)G_1X^*(T)| ^{1+\eta}  \leq   K\left[ \dbE^{\dbQ}     Z^*(T)^{m(1+\eta)} \right]^{\frac{1}{m}} \left[ \dbE^{\dbQ}     X^*(T)^{n(1+\eta)} \right]^{\frac{1}{n}}  < \infty. 
$$
Therefore, by the standard BSDE theory, BSDE (\ref{eq0208d}) admits a unique solution $(P, Q_1, Q_2)$ satisfying (\ref{eq0211d}). This completes the proof.

\end{proof}}

For simplicity of notation, we set
$$-\beta=2Z^*G X^*+A^{\top}P+ Z^* (q_2   H)^\top$$
and
$$-\alpha=\langle GX^*, X ^*\rangle +\langle Ru^*, u^*\rangle+q_2HX^*.$$
 Applying  It\^o's formula to $\langle P(\cdot), X^1(\cdot)\rangle$, we obtain 
\begin{align*}
d\langle P(t), X^1(t)\rangle =
&\( \langle \beta, X^1\rangle + \langle P, AX^1+Bv\rangle+\langle Q_1, Cv\rangle  \)    dt \\
&+\(\langle Q_1, X^1\rangle +\langle P, Cv\rangle \)dW(t)+\langle X^1, Q_2dY(t)\rangle.
\end{align*}
{\color{red}Note that $S(t)\equiv\int^t_0\<Q_1,X^1\>dW(s)$ is a local martingale. Let $\si_n$ be a localizing sequence. Then,  for $\eta \in (0, k)$, choose $m \in (1, \frac{1+k}{1+\eta})$, we have 
$$
\begin{aligned}
\dbE^{\dbQ}|S(t\wedge\si_n)|^{1+\eta}\le &  \dbE^{\dbQ}\(\sup_{s\in[0,T]}|X^1(s)|^{2}\int^T_0|Q_1(s)|^2ds\)^{(1+\eta)/2}\\
\leq &K\left\{\dbE^{\dbQ}\(\sup_{s\in[0,T]}|X^1(s)|^{(1+\eta)n} \) \right\}^{\frac{1}{n}}\left\{\dbE^{\dbQ}\(\int^T_0|Q_1(s)|^2ds\)^{m(1+\eta)/2} \right\}^{\frac{1}{m}}\\
<&   \infty. 
\end{aligned}
$$
Thus, $\{S(t\wedge\si_n): n\ge 1\}$ is uniformly integrable, and hence, $\dbE^{\dbQ} S(t)=0$. The other stochastic integrals can be calculated similarly.}
Therefore 
\begin{align*}
\mathbb{E}^\mathbb{Q}\langle P(T), X^1(T)\rangle = \mathbb{E}^\mathbb{Q}\int^T_0
\( \langle \beta, X^1\rangle + \langle P, AX^1+Bv\rangle+\langle Q_1, Cv\rangle  \)dt.
\end{align*}
Similarly, 
\begin{align*}
d(p(t)Z^1(t))
=&\( \alpha Z^1 +q_2 \( Z^1  (H  X^* )^{\top} +Z^*  (H  X^1 )^{\top} \)^{\top}\) dt+q_1 Z^1 dW(t)\\
&+\(Z^1  q_2 +p   \(Z^1   (H  X^* )^{\top} +Z^*  (H X^1 )^{\top}\) \)dY(t).
\end{align*}
\textcolor{red}{
	Applying Lemma \ref{sss42} and  by Doob's  inequality,  we have $\dbE^{\dbQ} [\sup_{s \in [0, T]}|Z^1(s)|^{1+k'}]< \infty$.  We then choose a localizing sequence  $\tau_n$, take $r \in (0, k')$ and  $l \in (1, \frac{1+k'}{1+r})$, and apply Lemma \ref{aaas} to obtain:
 $$
 \begin{aligned}
 \dbE^{\dbQ}\left|\int^{t\wedge \tau_n}_0q_1Z^1dW(t)\right| ^{1+r} \leq &  \dbE^{\dbQ} \left\{\sup_{s \in [0, T]} |Z^1(s)|^{1+r}\left|\int^{ T}_0|q_1(s)|^2ds\right| ^{\frac{1}{2}(1+r)} \right\} \\
 \leq & \left\{\dbE^{\dbQ} \sup_{s \in [0, T]} |Z^1(s)|^{l(1+r)} \right\}^ { \frac{1}{l}}\left\{  \dbE^{\dbQ}  \left(\int^{ T}_0|q_1(s)|^2ds\right) ^{\frac{l'}{2}(1+r)} \right\}  ^{\frac{1}{l'}}\\
 < &  \infty,
 \end{aligned}
 $$
 Therefore, $\left \{\int^{t\wedge \tau_n}_0q_1Z^1dW(t); n \geq 1\right\} $ is   uniformly integrable, which implies that $ \dbE^{\dbQ}\int^{ T}_0q_1Z^1dW(t)  = 0 $. Similarly, $\dbE^{\dbQ}\int^{ T}_0Z^1q_2dY(t)  = 0$. 
 Moreover, for $h \in (0, k)$ and $m \in (1,  \frac{1+k}{1+h})$, we have 
 $$
 \begin{aligned}
& \dbE^{\dbQ} \left|\int^T_0 p(t) Z^*  (H X^1 )^{\top} dY(t)  \right| ^{1+h} \\
 \leq &K  \dbE^{\dbQ} \left|\int^T_0 p(t) (X^1)^\top Z^*  (t) dY(t)  \right| ^{1+h}  \\
 \leq & K  \dbE^{\dbQ} \left(\int^T_0 |p(t)|^2|X^1(t)|^2 Z^*  (t) ^2dt \right)^{\frac{1}{2}(1+h)}   \\
 \leq & K  \dbE^{\dbQ} \left(\sup_{t \in [0, T]}  |p(t)|^{1+h} \sup_{t \in [0, T]}  |X^1(t)|^{1+h} \sup_{t \in [0, T]}  |Z^* (t)|^{1+h}\right)\\
 \leq & K \left\{ \dbE^{\dbQ} \left(\sup_{t \in [0, T]}  |p(t)|^{2(1+h)} + \sup_{t \in [0, T]}  |X^1(t)|^{2(1+h)} \right)^n \right \}^{\frac{1}{n}} \left\{ \dbE^{\dbQ} \left( \sup_{t \in [0, T]}  |Z^* (t)|^{m(1+h)}\right)  \right \}^{\frac{1}{m}}  \\
 <& \infty, 
 \end{aligned}
 $$
 Therefore, $\dbE^{\dbQ} \int^T_0 p(t) Z^*  (H X^1 )^{\top} dY(t)  = 0 $. Similarly, $\dbE^{\dbQ} \int^T_0 p(t) Z^1  (H X^* )^{\top} dY(t)  = 0 $. }

So we have
\begin{align*}
\mathbb{E}^\mathbb{Q}(p(T) Z^1(T)) =  \mathbb{E}^\mathbb{Q}\int^T_0
\( \alpha Z^1 +q_2\( Z^1  (H  X^* )^{\top} +Z^*  (H  X^1 )^{\top} \)^{\top}\)dt. 
\end{align*}
We  continue the calculation   \rf{Ju*} with
\begin{align*}
0\le&\ \mathbb{E}^\mathbb{Q} \int^T_0
\langle 2Z^*GX^*+\beta+A^{\top}P+Z^*(q_2 H)^\top, X^1\rangle dt\\
&\ +\mathbb{E}^\mathbb{Q}\int^T_0
\(\langle GX^*, X^*\rangle +\langle Ru^*, u^*\rangle + \alpha +q_2 HX^*\) Z^1dt\\
&\ +\mathbb{E}^\mathbb{Q}\int^T_0
 \< 2Z^*Ru^*+B^{\top}P+C^{\top}Q_1, v\> dt\\
=&\mathbb{E}^\mathbb{Q}\int^T_0
\<2Z^* Ru^*+B^{\top}P+C^{\top}Q_1, v\> dt.
\end{align*}
Define a stopping time $\tau$ as follows 
\begin{equation}\label{stopping}
	\tau=\inf\left\{t \in [0, T]:\;\int^t_0|u^*(s)|^2ds  = K^2_c   \right\},
	\end{equation}
	with the convention $\inf \emptyset  = T$. 
 If $\tau < T$, then for  $t < \tau $,  we have $\int^{t}_0|u^*(s)|^2ds<K^2_c.$   By taking  any $v(\cd) $,  we  obtain 
 \begin{equation}
 	u^*(t)=-\frac1{2}\(\mathbb{E}^\mathbb{Q}[Z^*(t)| \mathcal{F}^Y_t]R(t)\)^{-1}\big(B^{\top}(t)\mathbb{E}^\mathbb{Q}[ P(t) | \mathcal{F}^Y_t]+C^{\top}(t)\mathbb{E}^\mathbb{Q}[Q_1(t)| \mathcal{F}^Y_t]\big).
 	\end{equation}
 
 Furthermore, $u^*(t)=0$ for $t\in[\tau,T]$.
In summary, we have
 \begin{equation}\label{conds}
 	u^*(t)= \begin{cases}  -\frac1{2}\Big\{ \mathbb{E}^\mathbb{Q}[Z^*(t)| \cF^Y_t]R(t)\Big\}^{-1}
 	\mathbb{E}^\mathbb{Q}\big[ B^{\top}(t) P(t)+ C^{\top}(t) Q_1(t) \bigm| \cF^Y_t\big],  \text {\q if \q } t < \tau \land T, \\ 0,    \quad \quad  \quad \quad \quad \quad \quad \quad \quad   \quad \quad \quad \quad  \quad \quad \quad \quad  \quad \quad \quad \quad   \quad \q\q\  \      \text {\q if \q }     \tau \land T \leq t \leq T.\end{cases}
 	\end{equation}
Therefore, we have the following SMP for   Problem (O). 
\begin{theorem}\label{mainthm}
	Let Assumptions (A1-A2) hold.  Suppose that $u^*(\cdot)$ is an optimal control  for Problem (O). Then,
	 $u^*(\cdot)$  satisfies condition (\ref{conds}), where  
	 $\big( X(\cdot),   Z(\cdot), P(\cdot),Q_1(\cdot), Q_2(\cdot),  p(\cdot), q_1(\cdot), q_2(\cdot)\big)$ 
	 denotes the adapted solution  to  the following FBSDEs: for  $t \in [0, T]$,
	\begin{equation} \label{state410}
		\left\{
		\begin{aligned}
			dX(t)&=\big(A(t) X(t)+B(t) u^*(t)\big)dt+C(t)u^*(t)dW(t), \\
			dZ(t)&=Z(t)\big(H(t)  X(t)\big)^{\top}dY(t),\\
			X(0)&= x, \;  Z(0)=1, \\
		\end{aligned}
		\right.
	\end{equation}
	and
	\begin{equation}\label{eq0722b}
		\left\{
		\begin{aligned}
			dP(t)=&-\Big(2 ZG   X+A^{\top} P+  q_2 ZH \Big)dt+ Q_1(t)dW(t)+ Q_2(t)dY(t),\\
			d p(t)=&-\Big(X^\top GX  +u^{*\top} R  u^*+ q_2 H X\Big)dt + q_1(t)dW(t)+q_2(t)dY(t) \\
		P(T)=&2Z(T)G_1 X(T),\qquad  p(T)=X(T)^\top G_1  X(T),
		\end{aligned}
		\right.
	\end{equation}
\end{theorem}

{\color{red}Note that the BSDE (\ref{eq0722b}) is understood in the sense of Definition \ref{def0208a}.}

\subsection{\color{red}Near optimal control for Problem (SLQ)}

Finally, in this subsection we derive the  near optimal controls for Problem (SLQ) in the weak formulation.
We take $K_c=n$   and  denote 
  $$\cU^n_{ad} \deq  \Big\{ u(t) : [0, T] \times \Omega \to \mathbb{R}^d | u(t) \in \cF^Y_t \ \text{and} \  \int^T_0 |u(t)|^2dt\le n^2,\;\;\forall t\in[0,T]\Big\},$$
   and  the corresponding  weak admissible controls denoted by $\cU^n_w$.

Further, we  denote  the corresponding optimal control obtained from Theorem \ref{thm42}   by 
 $u^n (\cdot) \in \cU^n_{ad}$,  and the corresponding weak admissible control denoted by 
 $u^n_w(\cdot)  = (\Omega^n, \mathcal{F}^n,  \mathbb{F}^n, \mathbb{P}^n, W^n(\cdot), Y^n(\cdot), u^n(\cdot)) \in \cU^n_w$.

For any $\ep\in (0, 1)$,  let $\bar{u}(\cdot) \in \cU_{ad}$ be a control  such that
 	$$J\big(\bar{u}(\cdot)\big)<\inf_{u(\cdot)\in\cU^w_{ad}}J\big(u(\cdot)\big)+\frac{1}{2}\ep.$$
 	Let  $\bar{X}(\cdot)$ be the corresponding state trajectory, and $\bar{Y}(\cdot)$ be the corresponding  observation process. 	Let
 	$$\tau_n=\inf\left\{t\ge 0:\;\int^t_0|\bar u(s)|^2ds>n^2\right\}.$$
 	  Define 
 	$$
 	\bar{u}^n(t)  = \bar{u}(t) 1_{t < \tau_n}, 
 	$$
 	
The following lemma  will be helpful to derive the  near optimality of the Problem (SLQ). 
 	
 	\begin{lemma}\label{ll21}
 		$\bar{u}^n(t) \in  \cU_{ad}^n $.
 	\end{lemma}
 	\begin{proof}
 		It is clear that $\int ^T_0 |\bar{u}^n(t)|^2dt \leq n^2$. Now we prove the adaptiveness. 
 		Note that for the following system 
 		$$\left\{\begin{array}{ccl}
 			dX(t)&=&\big(A(t) X(t)+B(t)\bar{u}^n(t)\big)dt+C(t)\bar{u}^n(t)dW(t),\\
 			dY(t)&=&H(t)X(t)dt+ d\widetilde{W}(t), \\
 			X(0)&=&x, \; Y(0) =0,
 		\end{array}\right.$$
 	we have	for $t<\tau_n$,  $X(t)=\bar X(t)$ and $Y(t)=\bar Y(t)$. Therefore, $\bar{u}^n(t)$ is $\cF^{Y}_t$-measurable.
 	\end{proof}
 	
Note that we have the following
\begin{equation}\label{eq0208a}
\limsup_{n\to \infty} J(u^n(\cdot))  \leq \lim_{n\to \infty} J(\bar{u}^n(\cdot))  =  J (\bar{u}(\cdot))  < \inf_{u(\cdot)\in\cU^w_{ad}}J\big(u(\cdot)\big)+\frac{1}{2}\epsilon,
\end{equation}
 therefore, for $n > m$, where $m$ is   sufficiently large, we have that 
\begin{equation}
J(u^n(\cdot)) \leq \inf_{u(\cdot)\in\cU^w_{ad}}J\big(u(\cdot)\big)+\epsilon.
\end{equation}
In summary, we have the following theorem. 
\begin{theorem}
Let Assumptions (A1-A2) hold. {\color{red}For any $\epsilon >0$, let $n=n(\ep)$ be such that $J(u^n(\cdot))  \leq J (\bar{u}(\cdot))+\frac{1}{2}\epsilon$
whose existence is guranteed by (\ref{eq0208a}). Then, $\bar u_{\epsilon} (\cdot)\equiv u^{n(\ep)}(\cdot)$ is an $\ep$-optimal control of the Problem (SLQ).
 	Furthermore,  $\bar u_{\epsilon} (\cdot)$    satisfies  the system  ( \ref{conds}-\ref{eq0722b}) with $K_c=n(\ep)$. }
\end{theorem}

\section{Conclusion}\label{se5}

In this paper, we studied  the  near optimal controls of LQ problem with partial observation in the weak formulation.   
Our work is an extension of Sun and Xiong \cite{SX}. As they stated in the conclusion of their article,  their  approach works 
only for a specific structure of the state equation.  Namely,  the diffusion of the state equation does not involve the state
and the control.   In contrast, in our model, the  control  is contained in the diffusion term,  and our approach is entirely
different from theirs.  More precisely,  we have  used a conditional argument to make the Girsanov transformation applicable to our
bounded setting. We also obtained a moment of order higher than 1 for the auxiliary process resulted from the Girsanov
transformation by using Fernique's theorem.  The methods are still valid for the case where the expectation of the state process appears in the diffusion term of the state equation. Nevertheless,  we have only obtained the  near optimal controls so far, and our method does not apply to the case of state-dependent diffusion terms.   We hope to report some relevant results about
that  case  in our future publications.

{\bf Acknowledgement}: We would like to thank an anonymous referee and the AE for the constuctive suggestions which improve this article substantially.


\begin{thebibliography}{999}
	

	
	 \bibitem{Ben}
	Bensoussan, A., \& van Schuppen, J. H.  1985. Optimal control of partially observable stochastic
	systems with an exponential-of-integral performance index.
	{\em SIAM J. Control Optim, 23}, 599--613.
	
	
	\bibitem{Bodenon1950} 
	Bode, H.W., \&  Shannon, C. E. 1950. 
	A simplified derivation of linear least square
	smoothing and prediction theory. {\em Proceedings of the IRE 38}, 417-425.
	
	\bibitem{CNW}
 Chen,  T.,  Nie, T. \& Wu, Z. 2025. Indefinite linear-quadratic partially observed mean-field
game. {\em arXiv preprint arXiv:2508.01568.}


	
	\bibitem{DA}
	Davis, M. H. A. 1976. The Separation  principle in  stochastic control via Girsanov solution. {\em SIAM J. Control Optim,  14(1)}, 176-188.
	
	
	\bibitem{F1962}
	Florentin, J. J. 1962. Partial observability and optimal control. {\em Int. J. Electron, 13},  263-279.
	
	\bibitem{Fujisakikuni1972}
	Fujisaki, M., Kallianpur, G., \&  Kunita, H. 1972. Stochastic differential equations
	for the nonlinear filtering problem. {\em Osaka J. Math. 9}, 19-40.
	
	
	
	
	\bibitem{Kallianpur}
	Kallianpur, G.  1980. {\em Stochastic Filtering Theory. Applications of Mathematics,
	13}. Springer-Verlag, New York-Berlin, xvi+316 pp.
	
\bibitem{Huang-Wang-Xiong2009}
Huang, J., Wang, G., \& Xiong, J. 2009.
 A maximum principle for partial information backward stochastic control problems with applications,
{\em SIAM J. Control Optim., 48},  2106--2117.

\bibitem{Huang-Wang-Zhang2020}
Huang, P., Wang, G., \& Zhang, H. 2020.
A partial information linear-quadratic optimal control problem of backward stochastic differential equation with its applications,
{\em Sci. China Inf. Sci., 63},  1-13.


	
	 \bibitem{HH1975}
	Kuo, H. H. 1975. {\em Gaussian Measure on Banach Spaces. Lecture Notes in Math. 463},
	Springer, Berlin. 
	
					\bibitem{Kushner-62}
Kushner, H.  1962.
Optimal stochastic control,
{\em IRE Trans. Autom. Control, \bf 7}, 120--122.

\bibitem{LNW}
Li, M., Nie,T.  \& Wu, Z.  2023. Linear-quadratic large-population problem with partial
information: Hamiltonian approach and riccati approach. {\em SIAM Journal on Control
and Optimization}, 61(4):2114–2139.
	
	
	 \bibitem{Li} 
	 Li, X. J., \&   Tang, S. J. 1995. General Necessary Conditions for Partially Observed Optimal Stochastic Controls.
	  {\em J.  Applied Proba., 32(4)}, 1118-1137.
	  
	  \bibitem{PR}
	  Pardoux, E. \& Răşcanu, A. 2014.
	 Stochastic differential equations, backward SDEs, partial differential equations.
	 {\em  Stoch. Model. Appl. Probab., \bf 69}
Springer, Cham, xviii+667 pp.

\bibitem{Shi-Wu2010}
Shi, J., \& Wu, Z. 2010.
The maximum principle for partially observed optimal control of fully coupled forward-backward stochastic system,
{\em J. Optim. Theory Appl., 145},  543--578.

\bibitem{Snyder}
Snyder, D. L.,   Rhodes, I. B., \&   Hoversten, E. V. 1977. A separation theorem for stochastic sontrol problems with point-process observations.
 {\em Automatica, 13},  85-87.

\bibitem{SX}
Sun, J. R., \&    Xiong, J. 2023. Optimal Control of Partially Observed Linear Dynamical Systems. 
{\em SIAM J. Control Optim,  61(3)},  1231-1247.

\bibitem{Wang-Wang-Yan2021}
Wang, G., Wang, W., \& Yan, Z. 2021.
Linear quadratic control of backward stochastic differential equation with partial information,
{\em Appl. Math. Comput., 403},  126--164.


\bibitem{Wang}
Wang, G.,  Wu, Z., \& Xiong, J.  2013.  Maximum principles for forward-backward stochastic control
systems with correlated state and {\color{red}observation} noises. {\em SIAM J. Control Optim. 51},  
491--524.

\bibitem{Wang-Wu-Xiong2015}
Wang, G.,  Wu, Z., \& Xiong, J.   2015.
A linear-quadratic optimal control problem of forward-backward stochastic differential equations with partial information,
{\em IEEE Trans. Automat. Control, 60},  2904--2916.

\bibitem{Wang-Yu2012}
Wang, G., \& Yu, Z. 2012.
A partial information non-zero sum differential game of backward stochastic differential equations with applications,
{\em Automatica, 48},  342--352.



\bibitem{WH}
Wonham, W. M. 1968. On the separation theorem of stochastic control. {\em SIAM J. Control, 6},
 312--326.
 
 \bibitem{Wu-Zhuang2018}
Wu, Z., \& Zhuang, Y. 2018.
Linear-quadratic partially observed forward-backward stochastic differential games and its application in finance,
{\em Appl. Math. Comput., 321}, 577--592.




\bibitem{XiongJieyu2021}
Xiong, J.,   Xu, Z. Q., \&  Zheng, J. Y.  2021. Mean-variance portfolio selection under partial information with drift uncertainty. {\em
Quant. Finance,  21(9)},  1461–1473.

\bibitem{Xiongyu2007}
Xiong, J., \&  Zhou, X.Y.  2007.  Mean-variance portfolio selection under partial information. {\em SIAM J. Control Optim.,  46(1)},  156-175.

\bibitem{xiong}
Xiong, J.  2008. {\em An Introduction to {\color{red}Stochastic Filtering Theory}}. Oxford University
Press, London.

\bibitem{Yong2012x}
Yong, J. M., \& Zhou, X. Y.  2012. {\em Stochastic controls: Hamiltonian systems and HJB equations (Vol. 43)}. Springer Science, Business Media.



\bibitem{Zhang2024}
Zhang, J. Q., \& Xiong, J.  2024. Stochastic filtering under model ambiguity. {\em J. Math. Anal. Appl.
533}, Paper No. 128020, pp. 14.



\bibitem{Zakai1969}
Zakai, M. 1969. On the optimal filtering of diffusion processes. {\em Z. Wahrsch. Geb.
11}, 230-243.


\end{thebibliography}
\end{document}